\documentclass[12pt]{article}
\usepackage{euscript,amsmath, amssymb, amsfonts}
\usepackage{color}

\newcommand{\R}{{\cal R}}

\newcommand{\GGG}{{\mathbb C[G]}}

\newcommand{\be}{\begin{equation}}
\newcommand{\ee}{\end{equation}}
\newcommand{\bee}{\begin{eqnarray}}
\newcommand{\eee}{\end{eqnarray}}

\newcounter{theorem}%

\makeatletter \@addtoreset{theorem}{section}

\newcounter{lemma}

\makeatletter \@addtoreset{lemma}{section}

\newcounter{proposition}
\newcommand{\proposition}{\par\refstepcounter{theorem}
           {\bf Proposition
           \arabic{section}.%
           \arabic{theorem}. }}

\makeatletter \@addtoreset{proposition}{section}

\newcounter{corollary}

\makeatletter \@addtoreset{corollary}{section}

\newcounter{definition}

\makeatletter \@addtoreset{definition}{section}

\newenvironment{proof}[1][Proof]{\noindent\textsf{#1.\ }}
{\hfill {\small $\square$}}

\makeatletter \@addtoreset{equation}{section}

\begin{document}

\sloppy \title
 {
The number of independent Traces and Supertraces on the Symplectic
Reflection Algebra $H_{1,\nu}(\Gamma \wr S_N)$ }

\author
 {
 S.E. Konstein%
\thanks{ I.E. Tamm Department of Theoretical Physics,
          P.N. Lebedev Physical Institute, RAS
          119991, Leninsky prosp., 53, Moscow, Russia}
\thanks{E-mail: konstein@lpi.ru}
 ,
 I.V. Tyutin$^*$%
\thanks{Tomsk State Pedagogical University, Kievskaya St. 60, 634061 Tomsk, Russia}
\thanks{E-mail: tyutin@lpi.ru}     }

\date{
}

\maketitle
\thispagestyle{empty}

\begin{abstract}

Symplectic reflection algebra $ H_{1, \,\nu}(G)$ has a
$T(G)$-dimensional space of traces whereas, when considered as a
superalgebra with a natural parity, it has an $S(G)$-dimensional
space of supertraces.
The values of $T(G)$ and $S(G)$ depend on the symplectic reflection
group $G$
and do not depend on the parameter $\nu$.

In this paper, the values $T(G)$ and $S(G)$ are explicitly calculated for the groups
$G= \Gamma \wr S_N$,
where $\Gamma$ is a finite subgroup of $ Sp(2,\mathbb C)$.

\end{abstract}




\section{Introduction}
\label{page}

Let $V:=\mathbb C^{2N}$, let $G\subset Sp(2N,\mathbb C)$  be a
finite group generated by symplectic reflections. In \cite{KT2}, it
was shown that Symplectic Reflection Algebra $H_{1,\nu}(G)$ has
$T(G)$ independent traces, where $T(G)$ is the number of conjugacy
classes of elements without eigenvalue $1$ belonging to the group
$G\subset Sp(2N)\subset End(V)$, and that the algebra
$H_{1,\nu}(G)$, considered as a superalgebra with a natural parity,
has $S(G)$ independent supertraces, where $S(G)$ is the number of
conjugacy classes of elements without eigenvalue $-1$ belonging to
$G\subset Sp(2N)\subset End(V)$. Hereafter, speaking about spectrum,
eigenvalues and eigenvectors, the rank of an element of the group
algebra $\mathbb C[G]$ of the group $G$, etc., we have in mind the
representation of the group algebra $\mathbb C[G]$ in the space $V$.
Besides, we denote all the units in groups, algebras, etc., by 1,
and $c\cdot 1$ by $c$ for any number $c$.

There are two families of  groups generated by symplectic
reflections, see \cite{GS} and also \cite{HW}, \cite{BG},
\cite{Gor}:

Family 1):  $G$ is a complex reflection group acting on $\mathfrak H
\oplus \mathfrak H^*$, where $\mathfrak H$ is the space of
reflection representation. In this case, $G$ is a product of several
groups  from the following set of Coxeter groups
\be\label{list}
A_n, \,\,B_n,  \,\,C_n, \,\,D_n,  \,\,E_n,  \,\,F_n,  \,\,G_2,
\,\,H_n,  \,\,I_2(n).
\ee

Family 2):  $G=\Gamma \wr S_N$, which means here
 $G=\Gamma^N \rtimes S_N$
acting on $ (\mathbb
C^2)^N$, where $\Gamma$  is a finite subgroup of $Sp(2,\mathbb C)$.

For groups $G$ from the set (\ref{list}), the list of values $T(G)$ and $S(G)$ is given in
\cite{stek}.

\medskip

In this work, we give the values of $T(G)$ and $S(G)$ for the 2nd family.
Namely, we found the generating functions
\bee\label{th0}
t(\Gamma,x)&:= &\sum_{N=0}^{\infty}T(\Gamma \wr S_N)x^N,
\\
s(\Gamma,x)&:= &\sum_{N=0}^{\infty}S(\Gamma \wr S_N)x^N
\label{th00}
\eee
for each finite subgroup $\Gamma\subset
Sp(2,\mathbb C)$, see
Theorem \ref{th1}.

All needed definitions are given in the Section \ref{prel}; the
structure, conjugacy classes and
 characteristic polynomials of the groups $\Gamma \wr S_N$ are described
in Section \ref{secG}.

To include the case $N=0$ in consideration in formulas (\ref{th0}) -- (\ref{th00}),
it is natural to set $\Gamma \wr S_0:= \{E\}$, where $\{E\}$ is the group
containing only one element $E$, and, since $\dim V=0$, set
$H_{1,\nu}(\Gamma \wr S_0):= \mathbb C[\{E\}]$.

Applying the definitions given in Section \ref{prel} to the algebra $H_{1,\nu}(\Gamma \wr S_0)$ we deduce that

a) if the algebra $H_{1,\nu}(\Gamma \wr S_0)= \mathbb C[\{E\}]$ is considered as superalgebra,
it has only a trivial parity $\pi\equiv 0$;

b) the algebra $H_{1,\nu}(\Gamma \wr S_0)= \mathbb C[\{E\}]$ has 1-dimensional space of traces
and 1-dimensional space of supertraces; these spaces coincide;

c) it is natural to set $T(\Gamma \wr S_0)= S(\Gamma \wr S_0)=1$;

d) the algebra $H_{1,\nu}(\Gamma \wr S_0)= \mathbb C[\{E\}]$
contains two
Klein operators (i.e., elements satisfying conditions
(\ref{k1}) -- (\ref{k3})), namely, $E$ and $-E$.

\section{Preliminaries}\label{prel}

\subsection{Traces}

Let ${\cal A}$ be an associative superalgebra with parity $\pi$.
All expressions of linear algebra are given for homogenous elements only
and are supposed to be extended to inhomogeneous elements via linearity.

A linear complex-valued function $str$ on ${\cal A}$ is called a {\it supertrace} if
$$str(fg)=(-1)^{\pi(f)\pi(g)}str(gf) \ \mbox{ for all } f,g\in {\cal A}.$$

A linear complex-valued function $tr$ on ${\cal A}$ is called a {\it trace} if
$$tr(fg)=tr(gf) \ \mbox{ for all } f,g\in {\cal A}.$$

The element $K\in \cal A$ is called a {\it Klein operator}, if
\begin{eqnarray}\label{k1}
 &&\pi(K)=0,\\
 \label{k2}
 &&K^2=1,\\
 &&Kf=(-1)^{\pi(f)} f K  \qquad   \mbox{ for all } f \in {\cal A}.
 \label{k3}
\end{eqnarray}

Any Klein operator, if exists, establishes an isomorphism
between the space of traces  on ${\cal
A}$ and the space of supertraces on ${\cal
A}$.

Namely, if $f\mapsto tr(f)$ is a trace, then $f\mapsto
tr(fK^{1+\pi(f)})$ is a supertrace, and if $f\mapsto str(f)$ is a
supertrace, then $f\mapsto str(fK^{1+\pi(f)})$ is a trace.


\subsection{Symplectic reflection group}

Let $V={\mathbb  C}^{2N}$ be endowed with a non-degenerate
anti-symmetric
$Sp(2N)$-invariant bilinear form $\omega(\cdot,\cdot)$,
let the vectors $ e_i \in V$, where $i=1,\,...\,,\,2N$, constitute
a basis in $V$.

The matrix $(\omega_{ij}):=\omega(e_i,\,e_j)$ is anti-symmetric and non-degenerate.

Let $x^i$ be the coordinates of ${ x}\in V$, i.e.,
${x}= e_i\,x^i$. Then
$\omega({ x},\, y)=
\omega_{ij} x^i y^j$ for
any ${ x},\, y \in V$.
 The indices $i$ are lowered and raised
by means of the forms $(\omega_{ij})$ and $(\omega^{ij})$, where $\omega_{ij}\omega^{kj}=\delta_i^k$.

\begin{definition}
The element $R\in Sp(2N)\subset End V$ is called a {\it symplectic reflection}, if
$rank(R-1)=2$.
\end{definition}

\begin{definition}
Any finite subgroup $G$ of $Sp(2N)$ generated by a set of symplectic reflections
is called a {\it symplectic reflection group}.
\end{definition}

In what follows, $G$ stands for a symplectic reflection group, and
$\R$ stands for the set of all symplectic reflections in $G$.

Let $R\in \R$.
Set
\bee
V_R &:=& Im (R-1) \,, \\
Z_R &:=& Ker (R-1) \,.
\eee

Clearly, $V_R$ and $Z_R$ are symplectically perpendicular, i.e., $\omega(V_R,\,Z_R)=0$,
and  $V=V_R\oplus Z_R$.

So, let
$x=x_{{\phantom \,}_{V_R}}+x_{{\phantom \,}_{Z_R}}$ for any $x\in V$,
where $x_{{\phantom \,}_{V_R}}\in V_R$ and $x_{{\phantom \,}_{Z_R}}\in Z_R$.
Set
\be
\omega_R(x,y):=\omega (x_{{\phantom \,}_{ V_R}},\,y_{{\phantom \,}_{V_R}}).
\ee

\subsection{Symplectic reflection algebra (following \cite{sra})}

Let $ {\mathbb C}[G] $ be the {\it group algebra} of $ {G} $, i.e., the
set of all linear combinations $\sum_{g\in  {G} } \alpha_g \bar g$,
where $\alpha_g \in {\mathbb C}$.

 If we were rigorists, we would
write $\bar g$
to distinguish $g$ considered as an element of $ {G} \subset End(V)$
from the same element $\bar g \in  {\GGG}$
considered as an element of the group algebra.
The addition in $ {\GGG} $ is defined as follows:
$$
\sum_{g\in  {G} } \alpha_g \bar g + \sum_{g\in  {G} } \beta_g \bar g
= \sum_{g\in  {G} } (\alpha_g + \beta_g) \bar g
$$
and the multiplication is defined by setting
$\overline {g_1\!}\,\, \overline {g_2\!} = \overline {g_1 g_2}$.
In what follows, however, we abuse notation and omit the bar sign over elements of the group algebra.

Let $\eta$ be a function on ${\R}$, i.e., a set of constants $\eta_R$ with $R\in\R$ such that
$\eta_{R_1}=\eta_{R_2}$ if $R_1$ and $R_2$ belong to
one conjugacy class of $ {G} $.

\begin{definition}
\label{defpage}
The algebra $H_{t,\eta}(G)$, where $t\in \mathbb C$, is an associative algebra with unit  { 1};
it is the algebra
${\mathbb C}[V]$ of (noncommutative) polynomials in the elements of $V$ with coefficients
in the group algebra ${\mathbb C}[G]$ subject to the relations
\bee
g x&=&g(x) g %
\mbox{ for any } g\in  {G}
                   \mbox{ and } x \in V,
                   \mbox{ where } g(x)= e_i g^i_j x^j
                   \mbox{ for }x=e_i x^i,\\
\label{rel}
 \!\!\!\!\!\!\!\!\!\!\!\! [ x  , y] &=& t \omega(x,y)
 +
       \sum_{R\in\R} \eta_R
\omega_R(x,y)R
\mbox{ for any  $x,y\in V$}.
\eee

The algebra $H_{t,\eta}(G)$ is called a {\it symplectic reflection algebra},
see \cite{sra}.
\end{definition}

The commutation relations (\ref{rel}) suggest
to define the {\it parity} $\pi$ by setting:
\be\label{2.6}
\pi (x)=1,\ \pi (g)=0
\ \mbox{ for any }x\in V, \mbox{ and }g\in G,
\ee
enabling one to consider $H_{t,\eta}(G)$ as an associative {\it superalgebra}.

We consider the case $t\ne 0$ only, for any such $t$ it  is equivalent to the case $t=1$.

Let ${\cal A}$ and ${\cal B}$ be superalgebras such that ${\cal A}$ is a ${\cal B}$-module.
We say that the superalgebra
${\cal A}\ast {\cal B}$ is a {\it crossed product} of ${\cal A}$ and ${\cal B}$ if
${\cal A}\ast {\cal B}={\cal A} \otimes {\cal B}$ as a superspace and
\[
(a_1\otimes b_1)\ast (a_2\otimes b_2)=a_1 b_1(a_2)\otimes b_1 b_2,
\]
see \cite{pass}.
The element $b_1(a_2)$ may include a sign factor imposed by the Sign Rule, see \cite{Del}, p. 45.

The (super)algebra $ H_{1, \,\eta}(G)$
is a deform of the crossed product
 of the Weyl algebra $W_N$ and the group algebra
of a finite subgroup $G \subset Sp(2N)$
generated by symplectic reflections.

\subsection{
The number of independent traces and supertraces on the symplectic
reflection algebras }


\begin{theorem} (\cite{KT2})\label{main1}
{\it Let the symplectic reflection group $ {G} \subset End(V)$ have
$T_G$ conjugacy classes without eigenvalue $1$ and $S_G$ conjugacy
classes without eigenvalue $-1$.

Then the 
algebra
 $ H_{1, \,\eta}(G)$ has
$T(G)=T_G$ independent traces whereas $ H_{1, \,\eta}(G)$ considered
as a superalgebra, see (\ref{2.6}), has $S(G)=S_G$ independent
supertraces. }
\end{theorem}

\begin{proposition} Let $G$, $G_1$ and $G_2$ be symplectic reflection groups and $G=G_1 G_2$.
Then
$T(G_1 G_2)=T(G_1)T(G_2)$ and $S(G_1 G_2)=S(G_1)S(G_2)$.
\end{proposition}

\begin{proposition}\label{klein}
  If there exists a $K\in G$ such that $K |_V=-1$,
then $K$ is  a Klein operator.
\end{proposition}

\section{The group $\Gamma \wr S_N$}
\label{secG}

\subsection{Finite subgroups of $Sp(2,\mathbb C)$}

The complete list of the finite subgroups $\Gamma \subset Sp(2,\mathbb C)$
is as follows, see, e.g., \cite{raf}:

{
\[
{{%
\begin{tabular}{|c|c|c|c|}
\hline {\bf $\Gamma$} & {\bf Order} & {\bf Presence of $-1$} &
\begin{tabular}{c}
{\bf The number of} \\
{\bf conjugacy classes}
{\bf $C(\Gamma)$}%
\end{tabular}
\\
\hline Cyclic group $\mathbf Z_n :=\mathbb{Z}/n\mathbb{Z}$ \ \ \ & $n$ & if \ $n$ is even & $n$ \\
\hline Binary dihedral group $\mathcal{D}_{n}\ \ $ & $4n$ & yes & $n+3$ \\
\hline Binary tetrahedral group $\mathcal{T}$ & $24$ & yes & $7$ \\
\hline Binary octahedral group $\mathcal{O}$ & $48$ & yes & $8$ \\
\hline Binary icosahedral group $\mathcal{I}$ & $120$ & yes & $9$ \\
\hline
\end{tabular}%
}}
\]%
}

It is easy to see that each of these groups, except
$\mathbf Z_{2k+1}$ with integer $k$, has $C(\Gamma) -1$
conjugacy classes without $+1$ in the spectrum and
has $C(\Gamma) -1$
conjugacy classes without $-1$ in the spectrum.
The group
$\mathbf Z_{2k+1}$ with integer $k$ has $C(\mathbf Z_{2k+1}) -1$
conjugacy classes without $+1$ in the spectrum and it
has $C(\mathbf Z_{2k+1})$
conjugacy classes without $-1$ in the spectrum.


\subsection{Symplectic reflections in $\Gamma\wr S_N$ (following \cite{EM})}

Let $V=\mathbb{C}^{2N}$ and let the symplectic form $\omega $ have the shape
\begin{equation}
\omega :=\left(
\begin{array}{cccc}
\varpi  &  &  &  \\
& \varpi  &  &  \\
&  & \ddots  &  \\
&  &  & \varpi
\end{array}%
\right),\ \text{ where }\varpi =\left(
\begin{array}{cc}
0 & 1 \\
-1 & 0%
\end{array}%
\right).
\end{equation}

The elements of the group $\Gamma \wr S_N$ have the form of $N
\times N$ block matrix with $2 \times 2$ blocks.
Consider the following elements of $\Gamma \wr S_N$
\bee
&& (D_{g,i})_{kl}:=\left\{
\begin{array}{l}
g, \ \ \text{ if } k=l=i, \\
1, \ \ \text{ if } k=l\ne i, \\
0, \ \ \text{ otherwise, }
\end{array}%
\right.
\label{D}
\\
&&(K_{ij})_{kl}:=\left\{
\begin{array}{ll}
\delta_{kl}, &\text{if } k,l\ne i, \ k,l\ne j, \\
\delta_{ki}\delta_{lj} + \delta_{kj}\delta_{li},&\text{otherwise, }\\
\end{array}%
\right.
\label{K}
\\
&& S_{g,ij}:=D_{g,i}D_{g^{-1},j}K_{ij},
\eee
where $i,j=1,...,N$, $i\ne j$, $1\ne g\in \Gamma$.
It is clear that $K_{ij}=K_{ji}$ and $ S_{g,ij}$ = $ S_{g^{-1},ji}$.

The complete set of symplectic reflections in $\Gamma\wr S_N$ consists of
$D_{g,i}$, $K_{ij}$
and $S_{g,ij}$, where
$1 \leqslant i < j \leqslant N$ and $1\ne g\in\Gamma$.
This set generates the group $\Gamma\wr S_N$.

The symplectic reflections $K_{ij}$ and $S_{g,ij}$ lie in one conjugacy class for all $i\ne j$ and $g\ne 1$;
the elements $D_{g,i}$ ($g\ne 1$) and $D_{h,j}$ ($h\ne 1$) lie in one conjugacy class if
$g$ and $h$ are conjugate in $\Gamma$.
So, the algebra $H_{1,\nu}(\Gamma^N \wr S_N)$ depends on $C(\Gamma)$ parameters $\nu$,
where $C(\Gamma)$ is the number of conjugacy classes in $\Gamma$ including the class $\{1\}$.

\subsection{Conjugacy classes}

Further,
the elements of the group  $\Gamma \wr S_N$ can be represented in the form
$D\sigma$ where $D\in \Gamma^N$ is a diagonal
$N \times N$ block matrix, each block being a $2 \times 2$-matrix,
and $\sigma$  is $N \times N$ block matrix of permutation
 each block being a $2 \times 2$-matrix.

The product has the form:
\[
(D_{1}\sigma _{1})(D_{2}\sigma _{2})=D_{3}\sigma _{3}
\]%
where $\sigma _{3}=\sigma _{1}\sigma _{2}$ and  $D_{3}=D_{1}\sigma
_{1}D_{2}\sigma _{1}^{-1}.$


Fix an element $g_{0}=D_{0}\sigma _{0}$. Since the permutation $%
\sigma _{0}$ is a product of cycles, there exists a~ permutation $\sigma
^{\prime }$ such that
\bee\label{dec}
\sigma ^{\prime }\sigma _{0}(\sigma ^{\prime })^{-1}=\left(
\begin{array}{cccc}
c_{1} &  &  &  \\
& c_{2} &  &  \\
&  & ... &  \\
&  &  & c_{s}%
\end{array}%
\right),
\text{ where $c_{k}$ are the cycles of length $L_{k}$, $\sum_{k}L_{k}=N$},
\\  \label{dec2}
c_{k}=\left(
\begin{array}{cccccc}
0 & 1 & 0 & 0 & ... & 0 \\
0 & 0 & 1 & 0 & ... & 0 \\
0 & 0 & 0 & 1 & ... & 0 \\
... & ... & ... & ... & ... & ... \\
0 & 0 & 0 & 0 & ... & 1 \\
1 & 0 & 0 & 0 & ... & 0%
\end{array}%
\right) .
\eee

The element $\sigma ^{\prime }D_{0}\sigma _{0}(\sigma ^{\prime })^{-1}$ has
the form%
\[
\sigma ^{\prime }D_{0}\sigma _{0}(\sigma ^{\prime })^{-1}=\left(
\begin{array}{cccc}
D_{1}c_{1} &  &  &  \\
& D_{2}c_{2} &  &  \\
&  & ... &  \\
&  &  & D_{s}c_{s}%
\end{array}%
\right),
\]%
where $D_{k}$ is an $L_{k}\times L_{k}$ diagonal block matrix, each block being a $%
2\times 2$-matrix:%
\[
D_{k}=\left(
\begin{array}{cccc}
g_{1}^{k} &  &  &  \\
& g_{2}^{k} &  &  \\
&  & ... &  \\
&  &  & g_{L_{k}}^{k}%
\end{array}%
\right) ,\text{ \ }g_{i}^{k}\in \Gamma .
\]%
Next, consider diagonal block matrices $%
H_{k}=diag(h_{1}^{k},h_{2}^{k},...,h_{L_{k}}^{k})$ and the elements
\[
H_{k}D_{k}c_{k}H_{k}^{-1}=\left(
\begin{array}{ccccc}
h_{1}^{k}g_{1}^{k}(h_{2}^{k})^{-1} &  &  &  &  \\
& h_{2}^{k}g_{2}^{k}(h_{3}^{k})^{-1} &  &  &  \\
&  & h_{3}^{k}g_{3}^{k}(h_{4}^{k})^{-1} &  &  \\
&  &  & ... &  \\
&  &  &  & h_{L_{k}}^{k}g_{L_{k}}^{k}(h_{1}^{k})^{-1}%
\end{array}%
\right) c_{k}.
\]

For any element $h_{1}^{k}\in \Gamma $, one can choose
\[
h_{2}^{k}=h_{1}^{k}g_{1}^{k}, \ h_{3}^{k}=h_{2}^{k}g_{2}^{k},\,
...,\  h_{L_{k}}^{k}=h_{L_{k}-1}^{k}g_{L_{k}-1}^{k}
\]
such that
\[
H_{k}D_{k}c_{k}H_{k}^{-1}=\left(
\begin{array}{ccccc}
1 &  &  &  &  \\
& 1 &  &  &  \\
&  & 1 &  &  \\
&  &  & ... &  \\
&  &  &  & h_{1}^{k}g_{1}^{k}g_{2}^{k}...g_{L_{k}}^{k}(h_{1}^{k})^{-1}%
\end{array}%
\right) c_{k}.
\]

So, each conjugacy class of $\Gamma \wr S_N$ is described by the set of cycles in the decomposition
(\ref{dec}), (\ref{dec2})
of $\sigma_0$, where each cycle is marked by some conjugacy class
of $\Gamma$.

The cycle of length $r$ marked by the conjugacy class $\alpha$ of $\Gamma$
with representative $g_\alpha\in \Gamma$ has the shape:
\be\label{marked-cyclec}
A_{\alpha,\,r}:=\left(
\begin{array}{ccccc}
1 &  &  &  &  \\
& 1 &  &  &  \\
&  & 1 &  &  \\
&  &  & ... &  \\
&  &  &  & g_\alpha%
\end{array}%
\right) c^{r}=\left(
\begin{array}{cccccc}
0 & 1 & 0 & 0 & ... & 0 \\
0 & 0 & 1 & 0 & ... & 0 \\
0 & 0 & 0 & 1 & ... & 0 \\
... & ... & ... & ... & ... & ... \\
0 & 0 & 0 & 0 & ... & 1 \\
g_\alpha & 0 & 0 & 0 & ... & 0%
\end{array}%
\right),
\ee
where $\alpha=1,...,C(\Gamma)$ and $r=1,2,...$;
$A_{\alpha,\,r}$ is an $r\times r$ block matrix, each block being a $2\times 2$ matrix.

So, each element $g\in\Gamma\wr S_N$ is conjugate to the element of the shape
\be\label{marked-cyc}
\left(
\begin{array}{ccccc}
A_{\alpha_1,\,r_1} &  &  &  &  \\
& A_{\alpha_2,\,r_2} &  &  &  \\
&  & ... &  &  \\
&  &  &  & A_{\alpha_s,\,r_s}%
\end{array}%
\right)
\ee

It is convenient to describe the conjugacy class
of $\Gamma \wr S_N$ with representative  (\ref{marked-cyc})
by the set of  nonnegative integers $p_r^\alpha$,
where $r=1,\, 2,\,3,...\,$, and $\alpha=1,\,...\,,\,C(\Gamma)$, such that
\be\label{such}
\sum_{\alpha,\,r}r p_r^\alpha=N.
\ee
The value $p_r^\alpha$ for some conjugacy class is the number of cycles of length $r$
in the decomposition
(\ref{marked-cyc}) marked by the conjugacy class $\alpha$ of the group $\Gamma$.

The restriction (\ref{such}) can be omitted and can serve as definition of $N$ for each set of the numbers $p_r^\alpha$.

The number of conjugacy classes in $\Gamma \wr S_N$ is equal to
\be
C(\Gamma \wr S_N) = \sum_{p_r^\alpha:\, \sum_{\alpha,\,r}r p_r^\alpha=N} 1.
\ee
The generating function $c(\Gamma,x)$ of the number of conjugacy classes is defined as
\be
c(\Gamma,x):= \sum_{N=0}^\infty C(\Gamma \wr S_N)x^N
\ee
and is equal to
\be
c(\Gamma,x)=
\sum_{p_r^\alpha} x^{\sum_{\alpha,\,r}r p_r^\alpha}
= \sum_{p_r^\alpha =0}^\infty
\prod_{r=1}^\infty
\prod_{\alpha=1}^{C(\Gamma)}
(x^r)^{p_r^\alpha}=
\prod_{r=1}^\infty
\prod_{\alpha=1}^{C(\Gamma)}
\frac 1 {1-x^r}=\left(\Psi (x)\right)^{C(\Gamma)},
\ee
where $\Psi (x)$ is the Euler function
\be
\Psi (x):=\prod_{r=1}^\infty \frac 1 {1-x^r}
\ee

\subsection{Characteristic polynomials of conjugacy classes}

Before seeking the generating functions $t(\Gamma,x)$ and $s(\Gamma,x)$,
let us find the characteristic polynomial of the conjugacy class $g$  of $\Gamma\wr S_N$
identified by the set $p_r^\alpha$.

Let $P^{}_M(\lambda):=\det (M-\lambda)$ be the characteristic polynomial of the matrix $M$.
Then it is easy to see that
\[
P^{} _{A_{\alpha,\,r}}(\lambda )=\det (A_{\alpha,\,r}-\lambda )=\det
(g_\alpha-\lambda ^{r})=P^{} _{g_\alpha}(\lambda ^{r}),
\]
where the marked cycle $A_{\alpha,\,r}$ is defined by (\ref{marked-cyclec}).

Let $g\in\Gamma\wr S_N$ be defined by Eq. (\ref{marked-cyc}).

Now, it is easy to show that
\be\label{det}
P^{}_g(\lambda)=\det(g-\lambda)=
\prod_{i=1}^s \det( A_{\alpha_i,\, r_i}-\lambda ) =
\prod_{r,\alpha} (\det( g_\alpha-\lambda ^{r}) )^{p_r^\alpha},
\ee
if $g$ is a representative of the congugacy class in $\Gamma\wr S_N$ corresponding to the set
$p_r^\alpha$.

\begin{definition}
We call a conjugacy class {\it $t$-admissible}, if its representative $g\in \Gamma\wr S_N$
is such that $P^{}_g (1) \ne 0$.
\end{definition}
\begin{definition}
We call a conjugacy class {\it $s$-admissible}, if its representative $g\in \Gamma\wr S_N$
is such that $P^{}_g (-1) \ne 0$.
\end{definition}
\begin{definition}
We call a marked cycle $A_{\alpha,\, r}$, see Eq. (\ref{marked-cyclec}), {\it $t$-admissible},\\
if $P^{}_{A_{\alpha,\, r}} (1) \ne 0$.
\end{definition}
\begin{definition}
We call a marked cycle $A_{\alpha,\, r}$, see Eq. (\ref{marked-cyclec}), {\it $s$-admissible},\\
if $P^{}_{A_{\alpha,\, r}} (-1) \ne 0$.
\end{definition}


Recall that $g_\alpha\in\Gamma \subset Sp(2,\mathbb C)$, where $\Gamma$ is a finite group.
So $\det g_\alpha =1$ and the Jordan normal  form of $g_\alpha$ is diagonal.
This implies that if $g_\alpha$ has $+1$ in its spectrum,
then $g_\alpha=1$ and if $g_\alpha$ has $-1$ in its spectrum,
then $g_\alpha=-1$.
These facts together with Eq. ({\ref{det}}) imply, in their turn,
the following two propositions:
\begin{proposition}\label{pr1}\it
The
conjugacy class 
of $\Gamma\wr S_N$
identified by the set $p_r^\alpha$
is $t$-admissible
 if and only if $g_\alpha\ne 1$ for all $\alpha$, $r$
with $p_r^\alpha\ne 0$.
\end{proposition}

\begin{proposition}\label{pr2}\it
The conjugacy class
of $\Gamma\wr S_N$
identified by the set $p_r^\alpha$
is $s$-admissible
 if and only if
for any pair $\alpha,\,r$ such that $p_r^\alpha\ne 0$,
at least one of the next three conditions holds:
\begin{eqnarray*}
&&\text{ a) $g_\alpha\ne -1$ and $g_\alpha\ne 1$,}\\
&&\text{ b) $r$ is even and $g_\alpha = - 1$,}\\
&&\text{ c) $r$ is odd and $g_\alpha = 1$.}
\end{eqnarray*}
\end{proposition}

Note that the three sets of pairs $(r,\,\alpha)$ defined by the cases a), b), c)
in Proposition \ref{pr2} have empty pair-wise intersections.

Let $t_r(\Gamma)$ for $r=1,2...$ be equal to the number of different $\alpha$ such that $A_{\alpha,\, r}$
is $t$-admissible.

Evidently,
\be\label{tr}
t_r(\Gamma)=C(\Gamma)-1.
\ee

Analogously,  let $s_r(\Gamma)$ for $r=1,2...$ be  equal to the number of different $\alpha$ such that
$A_{\alpha,\, r}$
is $s$-admissible.

Evidently, if $\Gamma\ni -1$, then
\be\label{sr}
s_r(\Gamma)=C(\Gamma)-1.
\ee
and if $\Gamma\not\ni -1$, then
\be\label{srm}
s_{r}(\Gamma )=\left\{
\begin{array}{ll}
C(\Gamma )-1,&\text{ \ if }r\text{ is even}, \\
C(\Gamma ),&\text{ \ if }r\text{ is odd}.%
\end{array}%
\right.
\ee

\section{Combinatorial problem}

Consider the following simple combinatorial problem
(analogous problems are considered in \cite{hall}).

Suppose we have an unlimited supply of 1-gram colored weights for each of $n_1$ different colors,
an unlimited supply of 2-gram colored weights for each of $n_2$ different colors,
an unlimited supply of 3-gram colored weights for each of $n_3$ different colors, and so on.
Let $a^N_{n_1,\,...,\,n_k,\,...}$
be the number of opportunities to choose weights from our set of total mass $N$ grams.

The problem is to find generating function
\be
F_{n_1,...,n_k,...}(x) := \sum_{N=0}^\infty a^N_{n_1,\,...,\,n_k,\,...} x^N.
\ee

This problem is exactly the problem we discussed earlier. Namely, now we say \lq\lq $r$-gram weight"
instead of cycle of length $r$, and \lq\lq the number of different colors $n_r$" instead of
the number $t_r$ (\ref{tr}) or  $s_r$ (\ref{sr})--(\ref{srm}) of different $\alpha$.

\proposition\label{cp}
{\it
\be
F_{n_1+m_1,\, n_2+m_2,\,...,\,n_k+m_k,\,...}(x)
=
F_{n_1,\, n_2,\,...,\,n_k,\,...}(x) \cdot
F_{m_1,\, m_2,\,...,\,m_k,\,...}(x).
\ee
}
\begin{proof}
To prove this proposition, it suffices to note that
\be
a^N_{n_1+m_1,\, n_2+m_2,\,...,\,n_k+m_k,\,...}
=
\sum_{M=0}^N a^M_{n_1,\, n_2,\,...,\,n_k,\,...} \cdot
a^{N-M}_{m_1,\, m_2,\,...,\,m_k,\,...}.
\ee
\end{proof}

Introduce the functions
\be
f_i := F_{n^i_1,\, n^i_2,\,...,\,n^i_k,\,...}, \qquad \text{where $n^i_k=\delta^i_k$.}
\ee
Then
\be
 f_i(x)=1+x^i+x^{2i}+x^{3i}+\,...\,= \frac 1 {1-x^i}.
\ee

The next theorem follows from Proposition \ref{cp}
\begin{theorem}\label{comb}
\be
F_{n_1,\, n_2,\,...,\,n_k,\,...}=\prod_{i=1}^\infty (f_i)^{n_i}.
\ee
\end{theorem}

The function $F_{1,\,1,\,1,\,...}=\Psi(x)$ is the well-known Euler function,
the generating function of the number of partitions of $N$ into the sum of positive integers.

\section{Generating functions $t(\Gamma)$ and $s(\Gamma)$}

\begin{theorem}\label{th1}\it Set
\bee
&&\Psi(x)=\prod_{i=1}^\infty \frac 1 {1-x^i} \qquad \text{ (Euler function)},\\
&&\Phi(x):=\prod_{k=0}^\infty \frac 1 {1-x^{2k+1}}.
\label{phi}
\eee

Let $T(\Gamma\wr S_N)$ be the dimension of the space of traces on $H_{1,\eta}(\Gamma\wr S_N)$
and let
$S(\Gamma\wr S_N)$ be the dimension of the space of supertraces on $H_{1,\eta}(\Gamma\wr S_N)$
considered as a superalgebra.

Let

\[
t(\Gamma,x) :=\sum_{N=0}^\infty T(\Gamma\wr S_N)x^N\text{~~
and~~}s(\Gamma,x) :=\sum_{N=0}^\infty S(\Gamma\wr S_N)x^N.
\]
\noindent Then
\bee
&&\text{  $t(\Gamma,x)=(\Psi(x))^{C(\Gamma)-1}$},\\
&&\text{  $s(\Gamma,x)=(\Psi(x))^{C(\Gamma)-1}$,\ \ \ \ \ \ \ \
if $\Gamma \ne \mathbf Z_{2k+1}$},\\
&&\text{  $s(\Gamma,x)=(\Psi(x))^{C(\Gamma)-1}\Phi(x) $,\ \
if $\Gamma = \mathbf Z_{2k+1}$.}
\eee
\end{theorem}
\begin{proof}
To prove Theorem \ref{th1}, we apply Theorem \ref{comb}
to the numbers (\ref{tr}) -- (\ref{srm}) of admissible conjugacy classes.
It is clear that
\be
t(\Gamma)=F_{t_1(\Gamma),\,t_2(\Gamma),\,t_3(\Gamma),\,...}=\Psi^{C(\Gamma)-1},
\ee
\be
s(\Gamma)=F_{s_1(\Gamma),\,s_2(\Gamma),\,s_3(\Gamma),\,...}
=
\left\{
\begin{array}{ll}
\Psi^{C(\Gamma )-1},&\text{ \ if }\Gamma \ni -1, \\
\Psi^{C(\Gamma )-1}\Phi,&\text{ \ if }\Gamma \not\ni -1.
\end{array}
\right.
\ee
\end{proof}

Observe that $\Phi(x)=\sum_{i=0}^\infty O_N x_N$, where $O_N$ is the number of
partitions of $N$ into the sum of odd positive integers, and $O_N$ coinsides with
the number of independent supertraces on $H_{1\nu}(S_N)$, see \cite{KV}.

\subsection{Inequality theorem}

\begin{theorem}\label{in2}
{\it Let $G=\Gamma \wr S_N$.
For each positive integer  $N$, the following statements hold:
\bee
&&S(G)>0,\\
&&S(G)\geqslant T(G), \\
&& S(G) = T(G) \text{ if and only if } H_{1,\nu}(G) \text{ contains a Klein operator.}
\eee
}
\end{theorem}

Literally the same statements were proved for the Coxeter groups $G$ in \cite{stek},
and hence these statements hold for the product of any finite number of groups from
Family 1) and Family 2) defined on page \ref{page}.

\begin{proof}
Let $\Gamma \ne \mathbf Z_{2k+1}$.
Since each finite group $\Gamma\in Sp(2N,\mathbb C)$, except $\Gamma= \mathbf Z_{2k+1}$, contains $-1$,
the group $\Gamma\wr S_N$ contains Klein operator $K = \prod_{i=1}^N D_{-1,\,i}$.

There is no Klein operator in $H_{1,\nu}(\mathbf Z_{2k+1}\wr S_N)$ since for this algebra,
$S(\mathbf Z_{2k+1}\wr S_N)>T(\mathbf Z_{2k+1}\wr S_N)$, as it follows from
Theorem \ref{th1}.\end{proof}



\section*{Acknowledgments}
The authors (S.K. and I.T.) are grateful to Russian Fund for Basic Research
(grant No.~${\text{17-02-00317}}$)
for partial support of this work.


\end{document}